\documentclass[a4paper]{amsart}

\pdfoutput=1 

\usepackage{cmbright}
\usepackage{mathrsfs}
\usepackage{xcolor}
\usepackage[utf8]{inputenc}
\usepackage[T1]{fontenc}
\usepackage[english]{babel}
\usepackage{amsfonts}
\usepackage{amsthm}
\usepackage{amssymb}
\usepackage{mathcomp}
\usepackage{mathrsfs}
\usepackage[bb=boondox]{mathalfa}
\usepackage{graphicx}
\usepackage{epstopdf}
\usepackage[centertags]{amsmath}
\usepackage[pdftex]{hyperref}
\usepackage{vmargin}
\usepackage{tikz}
\usepackage{tikz-cd}
\usepackage{dsfont}
\usepackage{cases}
\usepackage{mathtools}
\usepackage{shuffle}
\usepackage{adjustbox}
\tikzcdset{scale cd/.style={every label/.append style={scale=#1},
    cells={nodes={scale=#1}}}}

\theoremstyle{plain}
\newtheorem{theorem}{Theorem}
\newtheorem*{theorem*}{Theorem}
\newtheorem{lemma}{Lemma}
\newtheorem{proposition}{Proposition}
\newtheorem*{proposition*}{Proposition}
\newtheorem{corollary}{Corollary}

\theoremstyle{definition}
\newtheorem{definition}{Definition}

\theoremstyle{remark}
\newtheorem{remark}{\sc Remark}
\newtheorem{example}{\sc Example}

\newcommand{\Ob}{\mathrm{Ob}}

\newcommand{\categ}[1]{\mathsf{#1}}

\newcommand{\operad}[1]{\mathsf{#1}}

\newcommand{\catofcog}[1]{\operad{#1}\mathrm{-}\mathsf{cog}}

\newcommand{\catofalg}[1]{\operad{#1}\mathrm{-}\mathsf{alg}}

\newcommand{\Fun}[2]{\mathrm{Fun}\left(#1,#2\right)}
\newcommand{\Funinfty}[2]{\mathrm{Fun}_{\infty}\left(#1,#2\right)}

\newcommand{\id}{\mathrm{Id}}

\newcommand{\II}{\mathbb{1}}

\newcommand{\Set}{\mathsf{Set}}
\newcommand{\sSet}{\mathsf{sSet}}

\newcommand{\hammock}{\mathrm{L}_{W}^{\mathrm{H}}}

\newcommand{\op}{\mathrm{op}}

\newcommand{\itemt}{\item[$\triangleright$]}

\newcommand{\poubelle}[1]{}

\title{Localisation of cubical model categories}
\author{Brice Le Grignou}
\email{bricelegrignou "at" gmail.com}

\date{\today}

\begin{document}

\begin{abstract}
    In this article we introduce the notion of a square structure on a model category, that generalises cubical model categories. We then show that under some homotopical conditions on this square structure the induced cubical category is a localisation of the model category.
\end{abstract}
\maketitle

\setcounter{tocdepth}{1}
\tableofcontents

\section*{Introduction}

Given an $\infty$-category $\categ C$ and a set morphisms $W$, the localisation $\mathcal{L}_W \categ C$ of $\categ{C}$ with respect to $W$ is the $\infty$-category obtained from $\categ C$ by inverting (up to homotopy) the morphisms in $W$.
Localisation in a key topic in the theory of higher categories as any $\infty$-category (to be more precise any $(\infty,1)$-category) may be obtained as the localisation of a category with respect to a sub-category.

Given a model category $\categ M$, one can compute its localisation using for instance the hammock localisation (\cite{DwyerKan80b}). However, the hammock mapping spaces may be huge. Nonetheless, for any two objects $X,Y$ one can describe the mapping space from $X$ to $Y$ as the simplicial set
$$
\mathrm{Map}_{\categ M}(X,Y) = \hom_{\categ{M}}(QX, Y_-)
$$ where $QX$ is a cofibrant replacement of $X$ and the simplicial object $\Delta[n] \mapsto Y_n$ is a Reedy fibrant replacement of the constant simplicial object $Y$. However, one cannot in general compose these mapping spaces, unless for instance $\categ M$ has the structure of a simplicial model category, that is a simplicial enrichment that satisfies some coherence with respect to the model structure.

Unfortunately, simplicial model categories are rare in the context of Homological Algebra. Despite that, one often has a natural cylinder object or a natural path object. Suppose that our model category has a natural cylinder object $X \mapsto C(X)$. Taking, the cylinder of the cylinder, the cylinder of the cylinder of the cylinder, \ldots  yields a monoidal functor from the category $\square$ of cubes to the category of endofunctors of $\categ M$
\begin{align*}
\square  &\to \Fun{\categ M}{\categ M}
\\
\square[n] &\mapsto C^n .
\end{align*}
We then obtain a cubical enrichment on $\categ M$, whose mapping spaces are
$$
\categ M^C(X,Y)_n = \hom_{\categ M}(C^n X, Y) .
$$
The resulting cubical category $\categ M^C$ (at least its restriction to fibrant-cofibrant objects) is a relevant candidate to be the localisation of $\categ{M}$.

We define a left square structure $Q$ on $\categ M$ to be an oplax monoidal functor from $\square$ to $\Fun{\categ M}{\categ M}$. This induces a cubical enrichment on $\categ M$ as just above. Such a square structure is called homotopical if for any cofibration $f : X \to X'$ and any fibration $g: Y \to Y'$, the map
$$
\categ M^C(X',Y) \to \categ M^C(X,Y) \times_{\categ M^C(X,Y')} \categ M^C(X,Y')
$$
is a fibration that is acyclic if $f$ or $g$ is. Moreover, we require the map $Q_0 \to \id$ to be an objectwise weak equivalence and $Q_0(\emptyset)$ to be cofibrant. A natural cylinder object whose induced square structure is homotopical is called a coherent cylinder.

Square structures are encountered in many contexts. Indeed, any enrichment of $\categ M$ by some monoidal model category $\categ E$ that satisfies some coherence conditions induces an homotopical square
structure on $\categ M$.  For instance, if
$\categ M$ is a simplicial model category. In the context of homological algebra where the model categories involved fail to be simplicial model categories, there are still homotopical square structure. We shall describe in particular a left square structure $Q$ on algebras over some planar operad so that $Q_0$ is always cofibrant.

Our final result assert that an homotopical square structure gives a model of the localisation of $\categ M$.

\begin{theorem*}
Let $Q$ be an homotopical left square structure on a model category $\categ M$. Then the functor of cubical categories
$$
\categ M_{cf} \to \categ M_{cf}^Q
$$
is a localisation with respect to weak equivalences.
\end{theorem*}

\subsection*{Layout}

In the first section we recall some notions about cubical sets, homotopy colimits and infinity-categories. The second section introduces the notion of a square structure and describes some 2-categories of categories equipped with such a structure. The next section describes the notion of an homotopical square structure and coherent cylinder. In particular, we give conditions for a natural cylinder to be coherent. The final section proves the result that an homotopical square structure on a model category $\categ M$ provides a model of the localisation of $\categ M$.

\subsection*{Notations and conventions}

\begin{enumerate}
    \item Let $\mathcal{U} < \mathcal{V}$ be two universes. A set will be called small if it is $\mathcal{U}$-small and large if it is $\mathcal{V}$-small.
    \item The category of small categories is denoted $\categ{Cat}$.
    \item A category is called (co)complete if it has all small (co)limits.
    \item In this article, a model category is a (large, locally small) category $\categ M$ with finite limits and colimits equipped with wide subcategories $W, \mathrm{Cof}, \mathrm{Fib}$ that contains all isomorphisms and so that 
    \begin{enumerate}
        \item any morphism may be factored as a cofibration followed by an acyclic fibration (resp. an acyclic cofibration followed by a fibration);
        \item cofibrations are maps which have the left lifting property with respect to acyclic fibrations;
        \item acyclic cofibrations are maps which have the left lifting property with respect to fibrations;
        \item cofibrations, fibrations and weak equivalences are stable by retracts;
        \item weak equivalences satisfy the 2-out-of-3 rule : if among three maps $f,g, f \circ g$, two of them are weak equivalences, then they are all weak equivalences.
    \end{enumerate}
    \item We denote $\Set$ the category of sets and $\sSet$ the category of simplicial sets. The context will specify if it means small or large (simplicial) sets.
    \item We denote $\mathrm{Ex}^\infty$ Kan's fibrant replacement of simplicial sets.
    \item A functor $f$ between small categories is called an homotopy equivalence if the morphism of simplicial sets $N(f)$ is a Kan-Quillen equivalence.
    \item For any natural integer $n$, we denote $[n]$ the poset
    $$
    0 < 1 < \cdots < n .
    $$
    \item We will denote $\mathrm{Iso}$ the groupoid with two objects and that is equivalent to $\ast$. 
\end{enumerate}


\section{Recollections}

\subsection{Cubes and cubical sets}

\begin{definition}
Let $\square$ be the category of cubes, that is the subcategory of posets whose
\begin{itemize}
    \itemt objects are posets
    $$
    \{0<1\}^n, \ n\in \mathbb N;
    $$
    \itemt morphisms are compositions of products of the maps 
    \begin{align*}
    \delta^{i} &: \ast \xrightarrow{i} \{0<1\} ,\  i\in \{0,1\}
    \\
    \sigma &: \{0<1\} \to \ast .
    \end{align*}
\end{itemize}
Let $\square_{\leq 1}$ be the full subcategory of $\square$ spanned by the object $0$ and $1$.
\end{definition}

\begin{definition}
A cubical set is a presheaf on $\square$. The category of cubical sets is denoted $\square-\Set$. Moreover, the image through the Yoneda embedding of $\{0<1\}^n$ is denoted $\square[n]$ for any natural integer $n \in \mathbb N$.
\end{definition}

\begin{proposition}
The category of cubes $\square$ inherits from the product of posets the structure of a strict monoidal category
$$
\square[n] \otimes \square[n] = \square[n+m]
$$
so that, for any monoidal category $\categ C$, the composite functor
$$
\mathrm{Fun_{\otimes}}\left( \square, \categ C \right) \to \Fun{\square}{\categ C} \to \Fun{\square_{\leq 1}}{\categ C}
$$
from the category of monoidal functors from cubes to $\categ C$ to the category of functors from $\square_{\leq 1}$ to $\categ C$ is an equivalence.
\end{proposition}

\begin{definition}
We will denote $\delta_k^i$ the map
$$
\square[n] = \square[k] \otimes \square[0] \otimes \square[n-k]
\xrightarrow{\id \otimes \delta^i \otimes \id}
\square[k] \otimes \square[1] \otimes \square[n-k]
= \square[n+1]
$$
and call it a coface. Similarly, we will denote $\sigma_k$ the map
$$
\square[n+1] = \square[k] \otimes \square[1] \otimes \square[n-k]
\xrightarrow{\id \otimes \sigma \otimes \id}
\square[k] \otimes \square[0] \otimes \square[n-k]
= \square[n] 
$$
and call it a codegeneracy.
\end{definition}

\begin{proposition}
The category of cubical sets $\square-\Set$ inherits from cubes through Day convolution the structure of a bilinear monoidal category
$$
\square[n] \otimes \square[n] = \square[n+m]
$$
so that, for any cocomplete bilinear monoidal category $\categ C$, the composite functor
$$
\mathrm{Fun_{\otimes, cc}}\left( \square-\Set, \categ C \right) \to \mathrm{Fun_{\otimes}}\left( \square-\Set, \categ C \right)
\to \mathrm{Fun_{\otimes}}\left( \square, \categ C \right)
$$
from the category of cocontinuous monoidal functors from cubical sets to $\categ C$ to the category of monoidal functors from cubes to $\categ C$ is an equivalence.
\end{proposition}

\begin{definition}
For any natural integer $n \geq 1$ and, let $\partial \square[n]$ be the subobject of $\square[n]$ so that the elements of $(\partial \square[n])_m$ are the maps from $\square[m]$ to $\square[n]$ that factorises through one of the cofaces
$$
\delta_k^i : \square[n-1] \to \square[n],\ 0\leq  k \leq n-1, \ i \in \{0,1\} .
$$
For $n=0$, we have $\partial \square[0] = \emptyset$.
\end{definition}

\begin{definition}
For any natural integer $n \geq 1$, any $1 \leq k \leq n-1$ and $\epsilon \in \{0,1\}$, let $\sqcap^{k,\epsilon}[n]$ be the subobject of $\square[n]$ so that the elements of $(\sqcap^{k,\epsilon}[n])_m$ are the maps from $\square[m]$ to $\square[n]$ that factorises through one of the cofaces
$$
\delta_a^b : \square[n-1] \to \square[n].
$$
where $(a,b) \neq (k,\epsilon)$.
\end{definition}

\begin{definition}
Let $L_\Delta : \square-\Set \to \sSet$ be the colimit preserving functor that sends $\square[n]$ to $\Delta[1]^n$. It has a right adjoint denoted $R^\Delta$.
\end{definition}

\begin{definition}
For any cubical set $X$, let $\square \downarrow X$ be the transposed Grothendieck construction of $X$, that is the category whose objects are pairs $(\square[n], x)$ of an object $\square[n] \in \square$ and an element $x \in X_n$. A morphism from $(\square[n], x)$ to $(\square[m], y)$ is the data of a morphism of cubes $\phi : \square[n] \to \square[m]$ so that $\phi(y) = x$. This defines a functor
$$
\square \downarrow - : \square-\Set \to \categ{Cat}.
$$
\end{definition}

\subsection{Homotopy theory of cubical sets}

\begin{theorem}\cite[Theoreme 8.4]{Cisinski06}
The category of cubical sets admits a monoidal proper combinatorial model structure whose cofibrations are monomorphisms and whose weak equivalences are morphism $f$ so that $L_\Delta(f)$ is a weak equivalence of simplicial sets or equivalently $N(\square \downarrow f)$ is a weak equivalence of simplicial sets. Sets of generating cofibrations and generating acyclic cofibrations are respectively
$$
\{\partial \square[n] \to \square[n]| n \in \mathbb N\}, \quad \{\sqcap^{k,\epsilon}[n] \to \square[n]| n \in \mathbb N, \ 0 \leq k \leq n-1,\ \epsilon\in \{0,1\}\}.
$$
Moreover, the adjunction $L_\Delta \dashv R^\Delta$ is a Quillen equivalence.
\end{theorem}

\begin{proposition}\cite[Proposition 17]{LeGrignou18}\label{propositioneqgrokan}
There exists a natural weak equivalence of simplicial sets from
$N(\square \downarrow X)$ to $L_\Delta (X)$.
\end{proposition}

\subsection{Homotopy colimits}

\begin{definition}
Given a diagram of simplicial sets $F :J \to \sSet$, its homotopy colimit $\varinjlim_J^h F$ is the geometric realisation of the simplicial object in simplicial sets
$$
n  \mapsto \coprod_{a_0 \to \cdots \to a_n \in N(J)_n} F(a_0) .
$$
\end{definition}

\begin{definition}
Given a diagram of cubical sets $F :J \to \sSet$, its homotopy colimit $\varinjlim_J^h F$ is the homotopy colimit of $F \circ L_\Delta$.
\end{definition}

\begin{definition}
Given a diagram of categories $F :J \to \categ{Cat}$, its Grothendieck construction is the category $\int_{x \in J} F(x)$ whose objects are pairs $(x,y)$ where $x \in J$ and $y \in F(x)$. Moreover, a morphism from $(x,y)$ to $(x', y')$ is the data of a morphism $f : x \to x'$ in $J$ and a morphism $F(f)(y) \to y'$ in $F(x')$. There is a canonical functor $\int_{x \in J} F(x) \to J$.
\end{definition}

\begin{definition}
Given a diagram of categories $F :J \to \categ{Cat}$, its transposed Grothendieck construction
$$
\int_{x \in J}^t F(x) = (\int_{x\in J} F(x)^\op)^\op .
$$
There is a canonical functor $\int_{x \in J}^t F(x) \to J^\op$.
\end{definition}

\begin{lemma}
By definition, we have for any cubical set $X$
$$
\square \downarrow X = \int^t_{n \in \square^\op} X_n .
$$
\end{lemma}

\begin{proof}
Clear.
\end{proof}

\begin{theorem}\cite{thomason1979}\label{thmthomason}
Given a diagram of categories $F :J \to \categ{Cat}$, there is a natural equivalence that relates $\varinjlim_J^h NF$ to $N(\int_J F)$.
\end{theorem}

\begin{lemma}\label{lemmatw}
For any category $\categ C$, if $\mathrm{Tw}(\categ C)$ is the twisted arrows category of $\categ C$, then the two following functors
$$
\categ C^\op \leftarrow \mathrm{Tw}(\categ C) \to \categ C
$$
are homotopy equivalences.
\end{lemma}

\begin{proof}
These functors have adjoints.
\end{proof}

\begin{corollary}\label{corgro}
Given a diagram of categories $F :J \to \categ{Cat}$, there is a natural chain of homotopy equivalences relating its Grothendieck construction to its transposed Grothendieck construction.
\end{corollary}

\begin{proof}
By Lemma \ref{lemmatw}, there is a span of objectwise homotopy equivalences relating $F$ to $F^\op$. This gives a span of homotopy equivalences relating $\int_{J} F$ to $\int_J F^\op$. Moreover, there is a span of homotopy equivalences relating $\int_{J} F^\op$ to $(\int_J F^\op)^\op$.
\end{proof}

\begin{corollary}\label{corollaryhomotopycolimit}
Given a diagram of categories $F :J \to \categ{Cat}$, there is a chain of natural equivalences that relates $\varinjlim_J^h NF$ to $N(\int_J^t F)$.
\end{corollary}

\begin{proof}
This is a direct consequence of Theorem \ref{thmthomason}
and Corollary \ref{corgro}
\end{proof}

\begin{corollary}\label{corollaryhomotopycolimitcube}
Given a diagram of cubical sets, $F :J \to \square-\Set$, there is a chain of natural equivalences that relates $\varinjlim_J^h L_{\Delta} F$ to $N(\int^t_{x \in J} \square \downarrow F(x))$.
\end{corollary}

\begin{proof}
This is a consequence of Corollary \ref{corollaryhomotopycolimit} and Proposition \ref{propositioneqgrokan}
\end{proof}

\subsection{Simplicial categories}

\begin{definition}
As in \cite{Lurie17}, we denote $\mathfrak C \dashv N$ the adjunction between simplicial sets and simplicial categories whose left adjoint $\mathfrak C$ sends the simplicial set $\Delta[n]$ to the Boardman--Vogt resolution of $[n]$.
\end{definition}

\begin{definition}
Given a simplicial category $\categ C$, we denote $h(\categ C)$ the category obtained by taking the $\pi_0$ of the mapping spaces. This defines a functor from simplicial categories to  categories that is left adjoint to the inclusion. Moreover, given a quasi-category $\categ C$, we will denote $h(\categ C)$ for $h \mathfrak C (\categ C)$.
\end{definition}

\begin{definition}
Given a simplicial category, we denote $\mathcal N(\categ C)$ the quasi-category obtained as the simplicial nerve of a fibrant replacement of $\categ C$:
$$
\mathcal N(\categ C) = N(\mathrm{Ex}^\infty(\categ C)) .
$$
\end{definition}

\begin{definition}
Let $\categ C$ be a cubical category. Then, we denote $L_\Delta(\categ C)$ the simplicial category obtained by taking the image through $L_\Delta$ of each mapping space. Then, we will denote respectively $h(\categ C)$ and $\mathcal{N}(\categ C)$ instead of $hL_\Delta(\categ C)$ and $\mathcal{N}L_\Delta(\categ C)$.
\end{definition}

\begin{theorem}\cite{Bergner07}
There exists a left proper combinatorial model structure on simplicial categories whose
\begin{itemize}
    \itemt weak equivalences are morphisms $f : C \to D$ so that $h(f)$ is an equivalence of categories and for any objects $x,y$, the map $C(x,y) \to D(f(x), f(y))$ is a weak equivalence of simplicial sets (for the Kan-Quillen model structure);
    \itemt fibrations are morphisms $f : C \to D$ so that $h(f)$ is an isofibration of categories and for any objects $x,y$, the map $C(x,y) \to D(f(x), f(y))$ is a Kan-Quillen fibration.
\end{itemize}
\end{theorem}

\begin{theorem}\cite{Joyal00}\cite{Lurie17}
There exists a cartesian monoidal combinatorial model structure on simplicial sets whose
\begin{itemize}
    \itemt cofibrations are degreewise injections;
    \itemt weak equivalences are morphisms $f$ so that $\mathfrak C(f)$ is a weak equivalence for the Bergner model structure;
    \itemt fibrant objects are quasi-categories.
\end{itemize}
Moreover, the adjunction $\mathfrak C \dashv N$ is a Quillen equivalence.
\end{theorem}


\section{Square structures and cubical model categories}

\subsection{Some 2-categorical notion}

\begin{definition}
We will call strict 2-category a category enriched in categories and 2-functor a functor between categories enriched in categories.
\end{definition}

\begin{definition}
Let us denote $\categ{CAT}$ the strict 2-category of categories.
\end{definition}

\begin{definition}
Given a strict 2-category $\categ C$,
we denote
\begin{itemize}
    \itemt $\categ C^\op$ the strict 2-category with the same objects as $\categ C$ and so that
    $$
    \categ C^\op (X,Y) = \categ C (Y,X);
    $$
    \itemt $\categ C^{-,\op}$ the strict 2-category with the same objects as $\categ C$ and so that
    $$
    \categ C^{-,\op} (X,Y) = \categ C (X,Y)^\op;
    $$
    \itemt $\categ C^{\op,\op}$ the strict 2-category with the same objects as $\categ C$ and so that
    $$
    \categ C^{\op,\op} (X,Y) = \categ C (Y,X)^\op.
    $$
We have  canonical isomorphisms
$$
\categ C^{\op,\op} = (\categ C^\op)^{-,\op}
= (\categ C^{-,\op})^\op.
$$
\end{itemize}
\end{definition}

\begin{definition}
Let $\categ{CCAT}$ be the strict 2-category of cubical categories, that is
\begin{itemize}
    \itemt whose objects are categories enriched in cubical sets;
    \itemt whose morphisms are functors of categories enriched in cubical sets;
    \itemt whose 2-morphisms between two morphisms $f,g : C\to D$ are morphisms 
    $$
    h : [1] \times C \to D
    $$
    so that $h_{|\{0\} \times C} =f$ and $h_{|\{1\} \times C} =g$, that is the data of maps $h(X) : f(X) \to g(X)$ in $D$ for any object $X \in C$, so that the following diagram commutes
    $$
    \begin{tikzcd}
    C(X,Y)
    \ar[r, "f"] \ar[d, "g"']
    & D(f(X), g(Y))
    \ar[d]
    \\
    D(g(X), g(Y))
    \ar[r]
    & D(f(X), g(Y))
    \end{tikzcd}
    $$
    for any two objects $X,Y \in C$.
\end{itemize}
\end{definition}

\begin{definition}
Let $\categ{SCAT}$ be the strict 2-category of simplicial categories, that is
\begin{itemize}
    \itemt whose objects are categories enriched in simplicial sets;
    \itemt whose morphisms are functors of categories enriched in simplicial sets;
    \itemt whose 2-morphisms between two morphisms $f,g : C\to D$ are morphisms 
    $$
    h : [1] \times C \to D
    $$
    so that $h_{|\{0\} \times C} =f$ and $h_{|\{1\} \times C} =g$.
\end{itemize}
\end{definition}

\subsection{Square structure on a category}

\begin{definition}
A left square structure on a category $\categ C$ is the structure of a oplax module over the associative algebra $\square$, that is the data of
\begin{itemize}
    \itemt a functor
    \begin{align*}
        \square &\to \Fun{\categ C}{\categ C};
        \\
        {\square[n]} &\mapsto Q_n
    \end{align*}
    \itemt a morphism
    $$
    \alpha_{n,m} : Q_{n+m}  \to Q_n \circ Q_m ,
    $$
    for any $n, m \geq 0$ that is natural with respect to $\square[n]$ and $\square[m]$ in the sense that for any pairs of morphisms of cubes $\phi :\square[n] \to \square[n']$ and $\psi: \square[m] \to \square[m']$, the following diagram commutes
    $$
    \begin{tikzcd}
    Q_{n +m}
    \ar[r, "\alpha_{n,m}"]
    \ar[d,"Q_{\phi \otimes \psi}"']
    & Q_n \circ Q_m
    \ar[d,"Q_{\phi} \circ Q_{\psi}"]
    \\
    Q_{n'+m'}
    \ar[r, "\alpha_{n',m'}"']
    & Q_{n'} \circ Q_{m'};
    \end{tikzcd}
    $$
    \itemt a morphism
    $$
    \beta : Q_0 \to \id_{\categ C};
    $$
    \itemt so that the following diagrams are commutative
    $$
    \begin{tikzcd}
    {Q_{n+m+l}}
    \ar[r] \ar[d]
    & {Q_n \circ Q_{m+l}}
    \ar[d]
    \\
    {Q_{n+m} \circ Q_{l}}
    \ar[r]
    & {Q_n \circ Q_m \circ Q_l}
    \end{tikzcd}
    \begin{tikzcd}
    Q_n Q_0
    \ar[rd, swap, "\id \circ \beta"]
    &Q_n
    \ar[d, equal]\ar[r,"\alpha_{0,n}"]
    \ar[l, swap, "\alpha_{n,0}"]
    & Q_0Q_n
    \ar[ld,"\beta \circ \id"]
    \\
    & Q_n
    \end{tikzcd}
    $$
    for any $n,m,l \geq 0$.
\end{itemize}
A left squared category is the data of a category equipped with a left square structure.
\end{definition}

\begin{definition}
Given a left square structure on a category $\categ C$, the transposed functor $Q^t$ is defined as follows
\begin{align*}
    \categ C &\to \Fun{\square}{\categ{C}}
    \\
    X & \mapsto (\square[n]\mapsto Q_n (X)) .
\end{align*}
\end{definition}

\begin{definition}
Let $(\categ C,Q)$ be a left square category and let $X$ be a cubical set. If $\categ C$ has colimits indexed by $\square \downarrow X$, then we denote
$Q_X$ the endofunctor of $\categ C$ defined by
$$
Q_X(A) = \varinjlim_{\square[n] \to X} Q_n(A) .
$$
\end{definition}

\begin{definition}
A right square structure $P$ on a category $\categ C$ is the data of a left square structure on $\categ C^\op$. More, concretely, this corresponds to a functor
\begin{align*}
        \square^\op &\to \Fun{\categ{C}}{\categ{C}};
        \\
        {\square[n]} &\mapsto P^{(n)},
    \end{align*}
a natural morphism
$$
\alpha_{n,m} : P^{(n)} P^{(n+m)} \to P^{(n+m)},
$$
for any $n, m \geq 0$ and a morphism
    $$
    \beta : \id \to P^{(0)};
    $$
that satisfy the conditions dual to those of a left square structure.
\end{definition}

\begin{definition}
A two sided square structure $(Q,P)$ on a category $\categ C$ is the data of a left square structure so that any functor $Q_n$ has a right adjoint $P^{(n)}$. This corresponds equivalently to a right square structure so that any functor $P^{(n)}$ has a left adjoint $Q_n$.
\end{definition}

\subsection{Square functor}

Let us consider two left squared categories $(\categ C,Q)$ and $(\categ D,O)$. 

\begin{definition}\label{defleftlax}
A left lax square functor from $(\categ C,Q)$ to $(\categ D,O)$ is the data of
\begin{itemize}
    \itemt a functor $F : \categ C \to \categ D$;
    \itemt morphisms of functors $\gamma_n : O_n \circ F \to F \circ Q_n$ that are natural with respect to $\square[n] \in \square$;
    \itemt so that the following diagrams commute
    $$
    \begin{tikzcd}
    O_{n+m} \circ F
    \ar[r] \ar[d]
    & F \circ Q_{n+m}
    \ar[dd]
    \\
    O_{n} \circ O_m \circ F
    \ar[d]
    \\
    O_{n} \circ F \circ Q_m
    \ar[r]
    & F \circ Q_n \circ Q_m
    \end{tikzcd}
    \quad
    \begin{tikzcd}
    O_0 \circ F
    \ar[r] \ar[d]
    & F
    \ar[d]
    \\
    F \circ Q_0
    \ar[r]
    & F .
    \end{tikzcd}
    $$
\end{itemize}
\end{definition}

\begin{definition}
A left oplax square functor from $(\categ C,Q)$ to $(\categ D,O)$ is the data of
\begin{itemize}
    \itemt a functor $F : \categ C \to \categ D$;
    \itemt natural transformations $\gamma_n : F \circ Q_n \to O_n \circ F$
\end{itemize}
that satisfy mutatis mutandis, the same conditions as in Definition \ref{defleftlax}.
\end{definition}

\begin{definition}
A left square functor from $(\categ C,Q)$ to $(\categ D,O)$ is the data of a lax left square functor $(F,\gamma)$ whose structural natural transformation $\gamma_n$ are isomorphisms. This is equivalently an oplax left square functor $(F,\gamma)$ whose structural natural transformation $\gamma_n$ are isomorphisms.
\end{definition}

\begin{definition}
Given two left lax square functor $(F, \gamma)$ and $(G, \xi)$ from $(\categ C,Q)$ and $(\categ D,O)$, a square natural transformation from $(F, \gamma)$ to $(G, \xi)$ is the data of a natural transformation $A: F \to G$ so that the following diagram commutes
$$
\begin{tikzcd}
O_n \circ F
\ar[r] \ar[d]
& F \circ Q_n
\ar[d]
\\
O_n \circ G
\ar[r]
& G \circ Q_n
\end{tikzcd}
$$
for any $\square[n] \in \square$. One can define similarly a square natural transformation between
left oplax square functor.
\end{definition}

\begin{definition}
We denote
\begin{itemize}
    \itemt $\categ{LSquare}_{\mathrm{lax}}$ the strict 2-category of left squared categories, left lax square functors and square natural transformations;
    \itemt $\categ{LSquare}_{\mathrm{oplax}}$ the strict 2-category of left squared categories, left oplax square functors and square natural transformations.
\end{itemize}
\end{definition}

Let us consider now two right squared categories $(\categ A,P)$ and $(\categ B,R)$. 

\begin{definition}
A right lax square functor from $(\categ A,P)$ to $(\categ B,R)$ is the data of a left oplax square functor from $\categ A^\op$ to $\categ B^\op$. Similarly, a right oplax square functor from $(\categ A,P)$ to $(\categ B,R)$ is the data of a left lax square functor from $\categ A^\op$ to $\categ B^\op$.
One can define accordingly square natural transformations between right lax square functors and between right oplax square functors. This yields the strict 2-categories $\categ{RSquare}_{\mathrm{lax}}$ and $\categ{RSquare}_{\mathrm{oplax}}$.
\end{definition}

\begin{proposition}
The involutive 2-functor from $\categ{CAT}$ to $\categ{CAT}^{-,\op}$ that sends a category $C$ to its oppposite $C^\op$ induces two involutive 2-functors
\begin{align*}
    \categ{LSquare}_{\mathrm{lax}} &\simeq \categ{RSquare}_{\mathrm{oplax}}^{-,\op};
    \\
    \categ{LSquare}_{\mathrm{oplax}} &\simeq \categ{RSquare}_{\mathrm{lax}}^{-,\op}.
\end{align*}
\end{proposition}

\begin{proof}
Straightforward with the definitions.
\end{proof}

\begin{definition}
Let $\categ{LSquare}_{\mathrm{lax}}^{ts}$ be the strict 2-category
\begin{itemize}
    \itemt whose objects are two-sided squared categories;
    \itemt whose mapping category from $C$ to $D$ is
    $$
    \categ{LSquare}_{\mathrm{lax}}^{ts}(C,D)
    =
    \categ{LSquare}_{\mathrm{lax}}(C,D)
    $$
    where $C$ and $D$ are seen as left squared categories.
\end{itemize}
One can define similarly $\categ{LSquare}_{\mathrm{oplax}}^{ts}$, $\categ{RSquare}_{\mathrm{lax}}^{ts}$ and $\categ{RSquare}_{\mathrm{oplax}}^{ts}$.
\end{definition}

\begin{proposition}
One has a canonical isomorphism $\categ{LSquare}_{\mathrm{lax}}^{ts} = \categ{RSquare}_{\mathrm{oplax}}^{ts}$.
\end{proposition}

\begin{proof}
It suffice to notice given a functor $F : C \to D$ between two categories equipped with two sided square structures $(C, Q, P)$ and $(D, O, R)$, there is a canonical one to one correspondence between left lax square structures on $f$ and right oplax square structures on $f$. Indeed, given a left lax square structure $O_n F \to F Q_n $, one gets a right oplax square structure as follows
$$
F P^{(n)} \to R^{(n)} O_n F P^{(n)} \to R^{(n)} F Q_n P^{(n)} \to R^{(n)} F .
$$
\end{proof}

\subsection{From squared categories to cubical categories}

\begin{definition}
Given a category $\categ C$ equipped with a left square structure $Q$, one gets a category enriched in cubical sets denoted $\categ C^Q$ whose objects are those of $\categ C$ and so that for any two objects $X,Y$
$$
\categ C^Q(X,Y)_n = \hom_{\categ C}(Q_n X, Y). 
$$
The composition is given as follows
$$
\categ C^Q(Y,Z) \otimes \categ C^Q(X,Y)
= \varinjlim_{f: Q_n Y \to Z} \square[n] \otimes \varinjlim_{g: Q_m X \to Y} \square[m]
\simeq \varinjlim_{(f,g)} \square[n+m] \to  \categ C^\square(X,Z)
$$
where the last map consists in sending an element $(f,g)$ to the composition 
$$
Q_{n+m} X \xrightarrow{\alpha_{n,m}} Q_n (Q_m(X)) \xrightarrow{Q_n(g)} Q_n(Y) \xrightarrow{f} Z.
$$
The units are given by the maps $\beta(X)$. Moreover, we have a canonical functor of cubical categories
$$
\categ C \to \categ C^Q ,
$$
that is the identity of objects and that sends a morphism $f :X \to Y$ to the element $f \circ \beta(X) \in \categ C^Q(X,Y)_0$.
\end{definition}

\begin{definition}
Given a category $\categ C$ equipped with a right square structure $P$, the related category enriched in cubical sets is denoted $\categ C^P$. If $\categ C$ is equipped with a two sided square structure $(Q,P)$, the related category enriched in cubical sets is denoted $\categ C^Q$, $\categ C^P$ or $\categ C^{P,Q}$.
\end{definition}

\begin{proposition}
The construction $(\categ C, Q) \mapsto \categ C^Q$ yields a 2-functor from $\categ{LSquare}_{\mathrm{lax}}$ to the 2-category of cubical categories $\categ{CCAT}$.
\end{proposition}

\begin{proof}
Such a 2-functor sends a left lax square functor $F : (C, Q) \to (D,O)$ to the functor of cubical categories $C^Q \to D^O$ with the same underlying function on objects and that acts on mapping cubical sets as
$$
\hom_{C}(Q_n X, Y) \xrightarrow{F} \hom_{D}(FQ_n X, FY) \to \hom_{D}(Q_n FX, FY).
$$
Moreover, given two left lax square functors $F,G : (C,Q) \to (D,O)$ and a 2-morphism $A : F \to G$, the maps $A(X) : F(X) \to G(X)$ define a 2-morphisms between the induced cubical functors.
\end{proof}

\begin{corollary}
The construction $(\categ C, P) \mapsto \categ C^P$ yields a 2-functor from $\categ{RSquare}_{\mathrm{oplax}}$ to the 2-category of cubical categories $\categ{CCAT}$.
\end{corollary}

\begin{proof}
Such a 2-functor is given as the composition
$$
\categ{RSquare}_{\mathrm{oplax}} = \categ{LSquare}_{\mathrm{lax}}^{-,\op}
\to \categ{CCAT}^{-,\op}\xrightarrow{\op} \categ{CCAT},
$$
where the last arrow is the 2-functor that sends a cubical category to its opposite.
\end{proof}

\begin{proposition}
The following diagram of strict 2-categories commutes
up to a natural isomorphism
$$
\begin{tikzcd}
\categ{LSquare}_{\mathrm{lax}}^{ts} 
\arrow[r, equal]
\ar[d]
& \categ{RSquare}_{\mathrm{oplax}}^{ts}
\ar[r]
& \categ{RSquare}_{\mathrm{oplax}}
\ar[d]
\\
\categ{LSquare}_{\mathrm{lax}}
\ar[rr]
&& \categ{CCAT} .
\end{tikzcd}
$$
\end{proposition}

\begin{proof}
Given a two sided squared category $(C, Q, P)$, this natural isomorphism is the canonical one that relates $C^Q$ to $C^P$.
\end{proof}

\subsection{Adjunctions of square structures}

\begin{definition}\label{definsquareadj}
A squared adjunction between two left squared categories $(\categ C, Q)$ and $(\categ D, O)$ is the data of an adjunction in the 2-category $\categ{LSquare}_{\mathrm{lax}}$, that is the data of
\begin{itemize}
    \itemt  a left lax square functor $L$ from $(\categ C, Q)$ to $(\categ D, O)$;
    \itemt  a left lax square functor $R$ from $(\categ D, O)$ to $(\categ C, Q)$;
    \itemt square natural transformations $\eta : \id \to RL$ and $\epsilon : LR \to \id$;
    \itemt so that the following maps are identities
    \begin{align*}
        &L \to LRL \to L;
        \\
        & R \to RLR \to R.
    \end{align*}
\end{itemize}
\end{definition}

\begin{proposition}\cite{Kelly74}
The data of a squared adjunction between the two left squared categories $(\categ C, Q)$ and $(\categ D, O)$ is equivalent to the data of
\begin{itemize}
    \itemt an adjunction $L\dashv R$ between the categories $\categ C$ and $\categ D$;
    \itemt the structure of a left square functor on $L$.
\end{itemize}
\end{proposition}

\begin{definition}
A squared adjunction between two right squared categories $(\categ C, Q)$ and $(\categ D, O)$ is the data of a squared adjunction between the two left squared categories $\categ C^\op$ and $\categ D^\op$. Equivalently, this is an adjunction in the strict 2-category $\categ{RSquare}_{\mathrm{oplax}}$.
\end{definition}

\begin{definition}
Given an adjunction between categories
$$
\begin{tikzcd}
C \ar[rr, shift left, "L"]
&& D \ar[ll, shift left, "R"]
\end{tikzcd}
$$
and a left square structure $Q$ on $C$, the transferred square structure on $D$ is the left square structure $O$ defined as
$$
O_n = L Q_n R
$$
and whose structure maps are
\begin{align*}
    \alpha_{n,m}&: O_{n+m} = L Q_{n+m} R \to
    L Q_{n} Q_m R \to L Q_{n} RL Q_m R ;
    \\
    \beta &: O_{0} = L Q_{0} R \to LR \to \id .
\end{align*}
Then, $R$ gets the structure of a left lax square functor given by
$$
Q_n R \xrightarrow{\eta (Q_n R)} RL Q_n R .
$$
This induces the structure of a left oplax square functor on $L$ given by 
$$
L Q_n  \xrightarrow{LQ_n(\eta)} L Q_n RL .
$$
\end{definition}

\subsection{From cylinders and paths to square structures}

\begin{definition}
A natural cylinder object on a category $\categ C$ is a functor
$$
Q_{\leq 1} : \square_{\leq 1} \to \Fun{\categ C}{\categ C}
$$
that sends $\square[0]$ to the identity $\id_{\categ C}$.
\end{definition}

\begin{proposition}
Let $\categ C$ be a category. The following categories are canonically equivalent
\begin{enumerate}
    \item the category of natural cylinder objects on $\categ C$;
    \item the category of strict monoidal functors from $\square$ to $\Fun{\categ C}{\categ C}$ ;
    \item the category of monoidal functors from $\square$ to $\Fun{\categ C}{\categ C}$;
    \item the full subcategory of
    $$
    \categ{LSquare}_{\mathrm{oplax}}(\categ C, \categ C) \times_{\Fun{\categ C}{\categ C}} \{\id\}
    $$
    spanned by the left square structures $Q$ so that the maps
    \begin{align*}
        \alpha_{n,m} &: Q_{n+m} \to Q_n \circ Q_m
        \\
        \beta & : Q_0 \to \id
    \end{align*}
    are isomorphisms.
\end{enumerate}
\end{proposition}

\begin{proof}
This follows from a straightforward check.
\end{proof}


\section{Square structures and homotopy}

The goal of this section is to describe model categories equipped with square structures that satisfies some coherence with respect to the model structure.

\subsection{Homotopical square structure on model category}

\begin{proposition}
Let $\categ M$ be a model category and let $Q$ be a left square structure on $\categ M$. The following assertions are equivalent.
\begin{enumerate}
    \item for any cofibration $f : X \to X'$ and any fibration $g : Y \to Y'$, the following map of cubical sets
$$
\categ C^Q(X',Y) \to \categ C^Q(X,Y)\times_{\categ C^Q(X,Y')} \categ C^Q(X',Y')
$$
is a fibration that is also acyclic if either $f$ or $g$ is acyclic;
\item for any cofibration of cubical sets $f : A \to B$ and any cofibration $g: X \to Y$ in $\categ M$, the map
$$
Q_B(X) \coprod_{Q_A(X)} Q_A(Y) \to Q_B(Y)
$$
is a cofibration that is also acyclic if either $f$ or $g$ is acyclic;
\item for any cofibration $f: X \to Y$ in $\categ M$ and any $n \geq 0$, the map
$$
Q_n(X) \coprod_{Q_{\partial\square[n]}(X)} Q_{\partial\square[n]}(Y) \to Q_n(Y)
$$
is a cofibration that is also acyclic if $f$ is acyclic; moreover, for any generating acyclic cofibration
$$
\sqcap^{i, \epsilon}[n] \to \square[n] 
$$
the morphism
$$
Q_{n}(X) \coprod_{Q_{\sqcap^{i, \epsilon}[n]}(X)} Q_{\sqcap^{i, \epsilon}[n]}(Y) \to Q_{n}(Y)
$$
is an acyclic cofibration;
\item the functor
$$
X \in \categ{M} \mapsto (Q_n(X))_{n \in \square} \in \Fun{\square}{\categ{M}} 
$$
sends cofibration to Reedy cofibrations and acyclic cofibrations to Reedy acyclic cofibrations; moreover, for any generating acyclic cofibration
$$
\sqcap^{i, \epsilon}[n] \to \square[n] 
$$
and any cofibration $f: X \to Y$ in $\categ M$, the morphism
$$
Q_{n}(X) \coprod_{Q_{\sqcap^{i, \epsilon}[n]}(X)} Q_{\sqcap^{i, \epsilon}[n]}(Y) \to Q_{n}(Y)
$$
is an acyclic cofibration.
\end{enumerate}
\end{proposition}

\begin{proof}
This follows from a straightforward checking.
\end{proof}

\begin{definition}
Let $\categ M$ be a model category. A left square structure $Q$ on $\categ M$ is called homotopical if the three following conditions are satisfied:
\begin{enumerate}
    \item for any cofibration $f : X \to X'$ and any fibration $g : Y \to Y'$, the following map of cubical sets
$$
\categ C^Q(X',Y) \to \categ C^Q(X,Y)\times_{\categ C^Q(X,Y')} \categ C^Q(X',Y')
$$
is a fibration that is also acyclic if either $f$ or $g$ is acyclic;
\item the map $Q_0(X) \to X$ is a weak equivalence for any $X \in \categ M$;
\item $Q_0(\emptyset)$ is cofibrant.
\end{enumerate}
\end{definition}

This implies in particular that $Q_0$ preserves cofibrant objects.

\begin{definition}
Let $\categ M$ be a model category. A right square structure $P$ on $\categ M$ is called homotopical if the related left square structure on $\categ{M}^\op$ is homotopical, that is if the three following conditions are satisfied:
\begin{enumerate}
    \item for any cofibration $f : X \to X'$ and any fibration $g : Y \to Y'$, the following map of cubical sets
$$
\categ C^P(X',Y) \to \categ C^P(X,Y)\times_{\categ C^P(X,Y')} \categ C^P(X',Y')
$$
is a fibration that is also acyclic if either $f$ or $g$ is acyclic;
\item the map $X \to P^{(0)}(X)$ is a weak equivalence for any $X \in \categ M$;
\item $P^{(0)}(\ast)$ is fibrant.
\end{enumerate}
\end{definition}

\begin{proposition}
Let $(\categ M, Q,P)$ be a model category equipped with a two-sided square structure. Let us suppose that the map $Q_0 \to \id$ is an isomorphism, or equivalently that the map $\id \to P^{(0)}$ is an isomorphism. Then $Q$ is homotopical if and only if $P$ is homotopical.
\end{proposition}

\begin{proof}
It just follows from the definitions.
\end{proof}

\subsection{Coherent cylinders and paths}

\begin{definition}
Let $\categ M$ be a model category. A natural cylinder object is called a coherent cylinder if the induced left square structure is homotopical.
\end{definition}

\begin{definition}
A coherent path object $P^{(\leq 1)}$ on a model category $\categ M$ is the data of a coherent cylinder on $\categ M^\op$.
\end{definition}

\begin{definition}
A coherent cylinder-path pair $(Q_{\leq 1},P^{(\leq 1)})$ of a model category $\categ M$ is the data of a coherent cylinder of $\categ M^\op$ and an adjoint path object, or equivalently a coherent path object and an adjoint cylinder object.
\end{definition}

\subsection{From enriched model categories to homotopical square structures}

Let $\categ E$ be a monoidal model category whose unit is cofibrant.

\begin{definition}
For any $\categ E$-enriched category $\categ C$, its underlying category $\categ{C}_{\Set}$ is the category with the same objects and whose mapping sets are
$$
\hom_{\categ C_{\Set}}(X,Y) = \hom_{\categ E}(\II, \categ C(X,Y))
$$
for any two objects $X,Y$.
\end{definition}

\begin{definition}
 A $\categ E$-model category is the data of an $\categ E$-enriched category $\categ C$ together with a model structure on the underlying category $\categ C_{\Set}$ (that has finite limits and colimits), so that for any objects $X,Y$, the functors
 \begin{align*}
     X' \in \categ C_\Set^\op &\mapsto \categ C(X', Y) \in \categ E
     \\
     Y' \in  \categ C_\Set &\mapsto \categ C(X, Y') \in \categ E
 \end{align*}
 have both a left adjoint. Moreover, we require that for any cofibration $f : A \to B$  and any fibration $g : X \to Y$ in $\categ C_{\Set}$, the map
 $$
 \categ C(B,X) \to \categ C(A,X)\times_{\categ C(A,Y)} \categ C(B,Y)
 $$
 is a fibration that is acyclic whenever $f$ or $g$ is acyclic.
\end{definition}

\begin{definition}
In the case where $\categ{E}$ is the monoidal model category of cubical sets $\square-\Set$, we call cubical model category a $\square-\Set$-model category.
\end{definition}

Let us consider an interval of $\categ E$, that is an object $J$ equipped with a cofibration followind by a weak equivalence
$$
\II \sqcup \II \hookrightarrow J \xrightarrow{\sim} \II
$$
that factorises $(\id_\II, \id_\II)$.
Such an interval induces a monoidal Quillen adjunction
$$
\begin{tikzcd}
\square-\Set
\ar[rr, shift left, "L_J"]
&& \categ E
\ar[ll, shift left, "R^J"]
\end{tikzcd}
$$
whose left adjoint $L_J$ sends $\square[n]$ to $J^{\otimes n}$ (see for instance \cite{LeGrignou18}).

Given such an interval, any $\categ E$-model $\categ C$ induces a cubical model category $\categ C_J$. The mapping cubes are
$$
X, Y \mapsto R^J \categ C(X,Y);
$$
and the tensorisation and the cotensorisation are given by the formulas
\begin{align*}
    A \boxtimes_{\square} X = L_J(A) \boxtimes_{\categ E} X;
    \\
    \langle X, A \rangle_{\square} = \langle X, L_J(A) \rangle_{\categ E};
\end{align*}
for any $X \in \categ C, A \in \square-\Set$.

Moreover, given a cubical model category $\categ C$, the underlying model category gets from the cubical structure  a canonical homotopical two sided square structure $(Q,P)$ whose maps $Q_0 \to \id \to P^{(0)}$ are identities.

\begin{proposition}
Let $\categ M$ be a bicomplete model category. Then a left square structure $Q$ is isomorphic to the left structure induced by a cubical structure if and only if the following conditions are satisfied
\begin{enumerate}
    \item $Q$ is homotopical;
    \item the map $Q_0 \to \id$ is an isomorphism;
    \item for any $n \geq 0$, the functor $Q_n$ has a right adjoint $P^{(n)}$.
\end{enumerate}
\end{proposition}

\begin{proof}
Straightforward.
\end{proof}

\subsection{Cocontinuous coherent cylinders}

Let $\categ{M}$ be a model category equipped with a cylinder object $Q_{\leq 1}$ and the induced left square structure $(Q_n)_n$.

\begin{lemma}
If $Q_1$ preserves finite colimits, then the functor from finite cubical sets to endofunctors of $\categ M$
$$
A \mapsto Q_A
$$
is monoidal.
\end{lemma}

\begin{proof}
One can first notice that any functor $Q_n$ ($n \in \square$) preserves finite colimits. For any two finite cubical sets $A,B$ and any $X \in \categ M$, we have a natural isomorphism
\begin{align*}
    Q_A Q_B(X)
    &= \varinjlim_{(\square[n], \phi)\in \square \downarrow A} Q_n(\varinjlim_{(\square[m], \psi)\in \square \downarrow B} Q_m(X))
    \\
    &\simeq
    \varinjlim_{(\square[n], \phi)\in \square \downarrow A}
    \varinjlim_{(\square[m], \psi)\in \square \downarrow B} Q_n(Q_m(X))
    \\
    & \simeq
    \varinjlim_{(\square[n], \phi)\in \square \downarrow A}
    \varinjlim_{(\square[m], \psi)\in \square \downarrow B} Q_{n+m}(X)
    \\
    & \simeq Q_{A \otimes B}(X) .
\end{align*}
A straightforward check shows that it defines the structure of a monoidal functor on $A \mapsto Q_A$.
\end{proof}

Let us consider two morphisms between finite cubical sets $g: U \to U'$ and $k : V \to V'$ and let us consider the pushout
$$
\begin{tikzcd}
U \otimes V
\ar[r] \ar[d]
& U' \otimes V
\ar[d]
\\
U \otimes V'
\ar[r]
& W.
\end{tikzcd}
$$
Let us consider the following assertions:
\begin{enumerate}
    \item for any cofibration $f: X \to Y$ in $\categ M$, the map
    $$
    Q_{U'}(X) \coprod_{Q_U(X)} Q_U(Y) \to Q_{U'}(Y);
    $$
    is a cofibration that is acyclic if $f$ is acyclic;
    \item for any cofibration $f: X \to Y$ in $\categ M$, the map
    $$
    Q_{U'}(X) \coprod_{Q_U(X)} Q_U(Y) \to Q_{U'}(Y);
    $$
    is an acyclic cofibration;
    \item for any cofibration $f: X \to Y$ in $\categ M$, the map
    $$
    Q_{V'}(X) \coprod_{Q_V(X)} Q_V(Y) \to Q_{V'}(Y);
    $$
    is a cofibration that is acyclic if $f$ is acyclic;
    \item for any cofibration $f: X \to Y$ in $\categ M$, the map
    $$
    Q_{V'}(X) \coprod_{Q_V(X)} Q_V(Y) \to Q_{V'}(Y);
    $$
    is an acyclic cofibration.
\end{enumerate}

\begin{lemma}\label{lemmabigcube}
Let us suppose that the cylinder object $Q_1$ preserves finite colimits.
\begin{itemize}
    \itemt If the assertions (1) and (3) just above are true, then for any cofibration $f : X \to Y$ of $\categ M$, the map
$$
Q_{U' \otimes V'}(X) \coprod_{Q_W (X)} Q_W (Y) 
\to Q_{U' \otimes V'}(Y)
$$
is a cofibration that is acyclic if $f$ is acyclic.
\itemt  If the assertions (1) and (4) just above are true or if the assertions (2) and (3) just above are true, then for any cofibration $f : X \to Y$ of $\categ M$, the map
$$
Q_{U' \otimes V'}(X) \coprod_{Q_W (X)} Q_W (Y) 
\to Q_{U' \otimes V'}(Y)
$$
is an acyclic cofibration.
\end{itemize}
\end{lemma}

\begin{proof}
Let us consider a map $f: X \to Y$ in $\categ C$ and the following cube diagram
$$
    \begin{tikzcd}
    & Q_{U}(Q_{V'} Y)
    \ar[rr]
    && Q_{U'} Q_{V'} (Y)
    \\
    Q_{U}( Q_V(Y))
    \ar[ru] \ar[rr]
    && Q_{U'} (Q_V(Y))
    \ar[ru]
    \\
    & Q_{U}( Q_{V'} X)
    \ar[rr]\ar[uu]
    && Q_{U'} Q_{V'}(X) \ar[uu]
    \\
    Q_{U} (Q_V(X))
    \ar[rr] \ar[uu] \ar[ru]
    && Q_{U'} (Q_V(X))
    \ar[ru] \ar[uu]    
    \end{tikzcd}
$$
Let us denote $Z$ the colimit of this diagram without its final object $Q_{U'} Q_{V'}(Y)$ and let us denote
$$
A= Q_{V'}(X) \coprod_{Q_V(X)} Q_V(Y) .
$$
We can notice that the pushouts of the bottom and top faces of the cube are respectively $Q_W(X)$ and $Q_W(Y)$.

The morphism from $Z$ to $Q_{U'} Q_{V'}(Y)$ may be rewritten in different ways which yield in particular to the following commutative square
$$
\begin{tikzcd}
Z 
\ar[r, equal]
\ar[d, equal]
& Q_{U'}(A) \coprod_{Q_U(A)}  Q_{U} Q_{V'} (Y)
\ar[r]
& Q_{U'} Q_{V'} (Y)
\\
Q_{U' \otimes V'} (X) \coprod_{Q_W(X)}  Q_{W}(Y)
\ar[rr]
&& Q_{U' \otimes V'} (Y)
\ar[u, "\simeq"']
\end{tikzcd}
$$
Let us suppose that the assertions (1) and (3) are true. If $f$ is a cofibration (resp. an acyclic cofibration), then, by hypothesis (3), the map $A \to Q_{V'}(Y)$ is a cofibration (resp. an acyclic cofibration). Then, by hypothesis (1), the right top horizontal map in this square diagram is a cofibration (resp. an acyclic cofibration). Hence, the bottom horizontal map is also a cofibration (resp. an acyclic cofibration).

Now, let us suppose that the assertions (1) and (4) are true. If $f$ is a cofibration, then, by hypothesis (4), the map $A \to Q_{V'}(Y)$ is an acyclic cofibration, and by hypothesis (1), the right top horizontal map in this square diagram is an acyclic cofibration. Hence, the bottom horizontal map is also an acyclic cofibration.

The case where the assertions (2) and (3) are true, follows from the same arguments.
\end{proof}

\begin{proposition}\label{propcharacterisationcoherent}
Let us suppose that $Q_1$ preserves finite colimits. Then $Q_{\leq 1}$ is coherent if and only if the following conditions are satisfied:
\begin{enumerate}
    \item for any cofibration $f : X \to Y$, the map
    $$
    Q_1(X) \coprod_{(X \sqcup X)} (Y \sqcup Y) \to Q_1(Y)
    $$
    is a cofibration that is acyclic if $f$ is acyclic
    \item for any cofibration $f : X \to Y$ and for any of the two coface morphisms from $\square[0]$ to $\square[1]$, the induced morphism
    $$
    Q_1 (X) \coprod_X Y \to Q_1(Y)
    $$
    is an acyclic cofibration.
\end{enumerate}
\end{proposition}

\begin{proof}
The only if direction is clear. Let us prove the if direction.

Let us consider the following assumptions that depends on a natural integer $n \geq 1$.
\begin{itemize}
    \item[$(C_n)$] For any cofibration $f : X \to Y$ in $\categ M$ the map
$$
    Q_n(X) \coprod_{Q_{\partial\square[n]}(X)} Q_{\partial\square[n]}(Y) \to Q_n(Y)
$$
is a cofibration that is acyclic if $f$ is acyclic.
    \item[$(AL_n)$] For any cofibration $f : X \to Y$ in $\categ M$ and for any $\epsilon \in \{0,1\}$, the map
$$
    Q_n(X) \coprod_{Q_{\sqcap^{n-1, \epsilon}[n]}(X)} Q_{\sqcap^{n-1, \epsilon}[n]}(Y) \to Q_n(Y)
$$
    is an acyclic cofibration.
    \item[$(A_n)$] For any cofibration $f : X \to Y$ in $\categ M$, any $\epsilon \in \{0,1\}$ and any $0 \leq i \leq n-1$, the map
$$
    Q_n(X) \coprod_{Q_{\sqcap^{i, \epsilon}[n]}(X)} Q_{\sqcap^{i, \epsilon}[n]}(Y) \to Q_n(Y) 
$$
    is an acyclic cofibration.
\end{itemize}

Let us prove $(C_n)$. It is true for $n=0$ and $n=1$. Then, using Lemma \ref{lemmabigcube} and hypothesis (1) combined with the fact that for any $n \geq 1$, the following diagram in cubical sets is a pushout
$$
\begin{tikzcd}
{\partial\square[n] \otimes \partial\square[1]}
\ar[r] \ar[d]
& {\square[n] \otimes \partial\square[1]}
\ar[d]
\\
{\partial\square[n] \otimes \square[1]}
\ar[r]
& {\partial\square[n+1]}
\end{tikzcd}
$$
a straightforward induction shows that $C_n$ is true for any $n \geq 0$.

Then, the assertions $(AL_n), n \geq 1$ follow, using Lemma \ref{lemmabigcube}, from $(C_{k})_{k \geq 0}$ combined with hypothesis (2) and the fact that the following square is a pushout
$$
\begin{tikzcd}
{\partial\square[n-1] \otimes \square[0]}
\ar[r] \ar[d]
& {\partial\square[n-1] \otimes \square[1]}
\ar[d]
\\
{\square[n-1] \otimes \square[0]}
\ar[r]
& {\sqcap^{n-1, \epsilon}[n]}
\end{tikzcd}
$$
for any $n \geq 1$, $\epsilon \in \{0,1\}$.

Finally, $(A_n)$ follows from Lemma \ref{lemmabigcube} combined with $(C_k)_{k \geq 0}$ and $(AL_k)_{k \geq 1}$ and the fact that the following diagram is a pushout 
$$
\begin{tikzcd}
{\sqcap^{i-1, \epsilon}[i] \otimes \partial\square[n-i]}
\ar[r] \ar[d]
& {\square[i] \otimes \partial\square[n-i]}
\ar[d]
\\
{\sqcap^{i-1, \epsilon}[i] \otimes \square[n-i]}
\ar[r]
& {\sqcap^{i-1, \epsilon}[n]}
\end{tikzcd}
$$
for any $n \geq 1$, $\epsilon \in \{0,1\}$ and $1 \leq i \leq n$.
\end{proof}

\begin{corollary}
Let us suppose that $Q_1$ preserves finite colimits. Then $Q_{\leq 1}$ is coherent if and only if the following conditions are satisfied:
\begin{enumerate}
    \item for any cofibration $f : X \to Y$, the map
    $$
    Q_1(X) \coprod_{(X \sqcup X)} (Y \sqcup Y) \to Q_1(Y)
    $$
    is a cofibration
    \item for any cofibration $f : X \to Y$ and for any of the two coface morphisms from $\square[0]$ to $\square[1]$, the induced morphism
    $$
    Q_1 (X) \coprod_X Y \to Q_1(Y)
    $$
    is a weak equivalence;
    \item for any acyclic cofibration $f: X \to Y$, the map $Q_1(f)$ is a weak equivalence.
\end{enumerate}
\end{corollary}

\begin{proof}
The only if direction is clear. Let us prove the if direction.

We can first notice that combining hypothesis (1) and hypothesis (3), we obtain the hypothesis (1) of Proposition \ref{propcharacterisationcoherent}, which implies (using the same arguments as in the proof of this proposition) that the functor $Q^t : \categ M \to \Fun{\square}{\categ{M}}$ sends (acyclic) cofibration to Reedy (acyclic) cofibrations.
This fact combined with hypothesis (2) implies that the hypothesis (2) of Proposition \ref{propcharacterisationcoherent} is satisfied. Hence $Q_{\leq 1}$ is coherent.
\end{proof}

\subsection{Squared Quillen adjunction}

\begin{definition}
A squared Quillen adjunction between be two model categories $\categ M, \categ N$ equipped with homotopical left (or right) square structures is a squared adjunction
$$
\begin{tikzcd}
\categ M
\ar[rr, shift left, "L"]
&& \categ N
\ar[ll, shift left, "R"]
\end{tikzcd}
$$
that is also a Quillen adjunction.
\end{definition}

\begin{example}
Let $\categ M$ be a model category equipped with an homotopical right square structure $P$.

Let us consider a small category $\categ J$. Let us suppose that the category $\Fun{\categ J}{\categ{M}}$ admits the projective model structure, whose fibrations and weak equivalences are morphisms of diagrams that are objectwise respectively fibrations and weak equivalences of $\categ M$.

Then, the model category $\Fun{\categ J}{\categ{M}}$ admits an homotopical right square structure $P_{\categ J}$ defined by
$(P_{\categ J}^{(n)}F)(X) = P^{(n)}(F(X))$ and the adjunction
$$
\begin{tikzcd}
\categ{M}^{\mathrm{Ob}(\categ J)}
\ar[rr, shift left]
&& \Fun{\categ J}{\categ{M}}
\ar[ll, shift left]
\end{tikzcd}
$$
becomes a squared Quillen adjunction. Moreover, if the category $\categ M$ has colimits indexed by $\categ J$ the adjunction
$$
\begin{tikzcd}
\Fun{\categ J}{\categ{M}}
\ar[rr, shift left,"\varinjlim_J"]
&& \categ M .
\ar[ll, shift left,"\mathrm{cst}"]
\end{tikzcd}
$$
is also a squared Quillen adjunction.
\end{example}

\subsection{The example of chain complexes}

Let $\mathcal A$ be a commutative ring. The category of chain complexes of $\mathcal A$-modules admits a model structure whose
\begin{itemize}
    \itemt weak equivalences are quasi-isomorphisms;
    \itemt fibrations are degreewise surjections;
    \itemt cofibrations are degreewise injections whose cokernel are projective.
\end{itemize}
It then forms a closed symmetric monoidal model category.

Let $J$ be the cellular model of the interval in chain complexes of $\mathcal A$-modules, that is
$$
\begin{cases}
J_0 = \mathcal{A} \cdot |0\rangle \oplus \mathcal{A} \cdot |1\rangle
\\
J_1 =\mathcal{A} \cdot |01\rangle
\\
d(|01\rangle) = |1\rangle - |0\rangle.
\end{cases}
$$

\begin{lemma}\cite{BergerFresse04}
Let $\operad{BE}$ be the Barratt-Eccles operad. There exists a functor from $\square_{\leq 1}$ to the monoidal category of $\operad{BE}$-coalgebras so that the composite functor
$$
\square_{\leq 1} \to \catofcog{BE} \to \categ{Ch}_{\mathcal A}
$$
is the cylinder
$$
\mathcal{A} \oplus \mathcal{A} \to J \to \mathcal{A}.
$$
\end{lemma}

\begin{lemma}
The category of chain complexes $\categ{Ch}_{\mathcal A}$ admits a two sided square structure $(Q,P)$ so that 
$$
\begin{cases}
Q_n(X) = J^{\otimes n} \otimes X ;
\\
P^{(n)}(X) = [J^{\otimes n}, X].
\end{cases}
$$
\end{lemma}

\begin{proof}
Straightforward.
\end{proof}

Now let $\operad O$ be an operad in chain complexes. One has an adjunction
$$
\begin{tikzcd}
\categ{Ch}_{\mathcal A}
\ar[rr, shift left, "T_O"]
&& \catofalg{O}
\ar[ll, shift left,"U^O"]
\end{tikzcd}
$$
relating $\operad O$-algebras to chain complexes.

\begin{lemma}\label{lemmapathalgebra}
Let us suppose that $\operad O$ is a retract of $\operad O \otimes_H \operad{BE}$ (where $\otimes_H$ denotes the Hadamard levelwise tensor product). Then the category of $O$-algebras in chain complexes admits a natural path object with an adjoint cylinder. Moreover the adjunction $T_O \dashv U^O$ is a squared adjunction.
\end{lemma}

\begin{proof}
Given a $\operad O$-algebra $A$, the chain complex
$$
P^{(1)}(A) = [J, A] \simeq A \otimes J^\vee
$$
has the structure structure of an $\operad O \otimes_H \operad{BE}$-algebra induced by the structure of an $\operad O$-algebra on $A$ and the structure of a $\operad{BE}$-algebra on $J^\vee$. Hence $A \otimes J^\vee$ has the structure of an $\operad O$-algebra, using the morphism of operads $\operad O  \to \operad O \otimes_H \operad{BE}$.
This gives us the natural path object $P^{(1)}$ and thus a right square structure. It is clear that the adjunction $T_O \dashv U^O$ is a squared adjunction.

Finally, $P^{(1)}$ has a left adjoint given as the coequaliser of two maps
$$
T_O (J \otimes \operad O \triangleleft (A)) \rightrightarrows T_O( J \otimes A) ;
$$
(where $\operad O \triangleleft (A) =  \bigoplus_n \operad O(n) \otimes_{\mathbb S_n} A^{\otimes n} = U^O T_O U^O(A)$)
defined as follows:
\begin{enumerate}
    \item the first map is just induced by the structural map $\operad O \triangleleft (A) \to A$;
    \item the second map is induced by the maps
    $$
    J \otimes \operad O(n) \otimes A^{\otimes n} \to
    J  \otimes \operad{BE}(n) \otimes \operad O(n) \otimes A^{\otimes n}
    \to
    J^{\otimes n} \otimes \operad O(n) \otimes A^{\otimes n} \simeq
    \operad O(n) \otimes (J \otimes A)^{\otimes n}.
    $$
\end{enumerate}
\end{proof}

\begin{proposition}\label{propositionch}
If $\operad O$ is a retract of $\operad O \otimes_H \operad{BE}$, then
there exists a combinatorial model structure on the category of $\operad O$-algebras transferred from that of chain complexes. Moreover, its natural path object from Lemma \ref{lemmapathalgebra} is coherent and the adjunction $T_O \dashv U^O$ is a squared Quillen adjunction.
\end{proposition}

\begin{proof}
Using standard results about transfer of model structures (\cite{Hovey99}),
in order to show that the category of $\operad O$-algebras admits such a model structure, it suffices to show that it admits a natural path object (in the model categorical sense). Such a path is provided by Lemma \ref{lemmapathalgebra}.
The fact that this path object is coherent proceeds from the fact that the construction $X \mapsto [J, X]$ provides a coherent path object of chain complexes.
\end{proof}

\begin{corollary}
If $\operad O$ is a planar operad, then
there exists a combinatorial model structure on the category of $\operad O$-algebras transferred from that of chain complexes. Moreover, it has a coherent cylinder-path objects. Then the adjunction $T_O \dashv U^O$ is a squared Quillen adjunction.
\end{corollary}

\begin{proof}
It suffices to apply the same arguments as in the proofs of Lemma \ref{lemmapathalgebra} and of Proposition \ref{propositionch}, knowing that $J$ has the canonical structure of an $\operad{uAs}$-coalgebra (where $\operad{uAs}$ stands for the planar operad that encodes unital associative algebras) and that $\operad{uAs}$ is the unit of the levelwise tensor product of planar operads.
\end{proof}

\subsection{Another left square structure on algebras}

Now, let us suppose that $\mathcal A = \mathbb K$ is a field and that $\operad O= \operad{As}$ is the associative planar operad (actually, what we say works for any planar operad). Using the coassociative coalgebra structure on $J$, the model category of $\operad O$-algebras in chain complexes has a coherent path object given by the convolution algebra
$$
P^{(1)}(A) = [J,A].
$$
Moreover, this path object has an adjoint coherent cylinder $Q_1$.

Koszul duality techniques (see \cite{LodayVallette12}) provides us with an adjunction $B^\vee \dashv B$ relating associative algebras to locally conilpotent coassociative coalgebras
$$
\begin{tikzcd}
\text{Conilpotent-coalgebras}
\ar[rr, shift left, "B^\vee"]
&&
\catofalg{As}
\ar[ll, shift left, "B"]
\end{tikzcd}
$$

\begin{proposition}\cite{LefevreHasegawa03}
There exists a model structure on conilpotent coalgebras transferred from that on associative algebras that is whose
\begin{itemize}
    \itemt cofibrations are maps $f$ so that $B^\vee(f)$ is a cofibration;
    \itemt weak equivalences are maps $f$ so that $B^\vee(f)$ is a quasi-isomorphism.
\end{itemize}
Then, the adjunction $B^\vee \dashv B$ is a Quillen equivalence.
\end{proposition}

\begin{lemma}
The category of conilpotent coalgebras admits a coherent cylinder. Then, the adjunction $B^\vee \dashv B$ is a squared Quillen equivalence.
\end{lemma}

\begin{proof}
The category of conilpotent coalgebras is tensored over coassociative coalgebras. Then, using the structure of a coalgebra on $J$, we get a cylinder defined by
$$
Q_1^{cog} (V) = J \otimes V 
$$
Then, one can notice that $B^\vee$ commutes with the cylinders, in the sense that we have a canonical isomorphism
$$
Q_1 B^\vee = B^\vee Q_1^{cog} .
$$
Since the natural cylinder $Q_1$ of associative algebras is coherent and since the model structure on conilpotent coalgebras is transferred through $B^\vee$, then $Q_1^{cog}$ is coherent. It is clear then that the adjunction $B^\vee \dashv B$ is a squared Quillen equivalence.
\end{proof}

\begin{definition}
Let
$(C_n)_{n \in \square}$ be the left square structure on $\catofalg{As}$ transferred from $Q_n^{cog}$ through the adjunction $B^\vee \dashv B$, that is
$$
C_n (A) =Q_n B^\vee B(A) = B^\vee Q_n^{cog} B(A) =  B^\vee (J^{\otimes n} \otimes B(A));
$$
for any associative algebra in chain complexes $A$, and any $n \geq 0$.
\end{definition}

\begin{proposition}
The left square structure $(C_n)_{n \in \square}$ on associative algebras in chain complexes is homotopical.
\end{proposition}

\begin{proof}
This is a direct consequence of the fact that both $B$ and $B^\vee$ preserves cofibrations and weak equivalences and the fact that the map $B^\vee B (A) \to A$ is a weak equivalence for any algebra $A$.
\end{proof}

This left square structure $(C_n)_n$ on $\catofalg{As}$ has two interesting (and related to each other) properties. On the one hand, for any algebra $A$, $C_0(A)$ is a cofibrant replacement of $A$. On the other hand, for any two weak equivalences $A' \to A$ and $B \to B'$, the map
$$
\catofalg{As}^C(A, B) \to \catofalg{As}^C(A', B')
$$
is an equivalence of fibrant cubical sets.


\section{Localisation of cubical model categories}

In this section, we show that, given a model category $\categ M$ equipped with an homotopical left square structure $Q$, the restriction to fibrant-cofibrant objects of $\categ M^Q$ is a localisation of $\categ M$ with respect to weak equivalences.

We first recall results of Dwyer-Kan localisation and then prove the result.

\subsection{What is localisation}

\begin{definition}
Let $\categ C$ be a quasi-category and $W$ be a subset of maps. The localisation of $\categ C$ with respect to $W$ is a quasi-category $\mathcal{L}_W \categ C$ equipped with an $\infty$-functor $\pi : \categ C \to \mathcal L_W \categ C$
so that
\begin{enumerate}
    \item $\pi$ sends morphisms in $W$ to equivalences;
    \item for any quasi-category $\categ D$, the functor
    $$
    \Funinfty{\mathcal{L}_W \categ C}{\categ D} \to \Funinfty{\categ C}{\categ D} \times_{\Fun{W}{\categ D}}
    \Fun{W}{\mathrm{Core}(\categ D)}
    $$
    is an equivalence of quasi-categories (where $\mathrm{Core}(\categ D)$ is the largest Kan complex contained in $\categ D$).
\end{enumerate}
One can equivalently define $\mathcal{L}_W \categ C$ using the fact that we have two homotopy pushouts in the $\infty$-category of $\infty$-categories
$$
\begin{tikzcd}
\coprod_{w \in W}N({[1]})
\ar[r] \ar[d]
& \categ C
\ar[d]
\\
\coprod_{w \in W}N(\mathrm{Iso}) \ar[r]
& \mathcal{L}_W \categ C
\end{tikzcd}
\quad
\begin{tikzcd}
\Tilde{W}
\ar[r] \ar[d]
& \categ C
\ar[d]
\\
\mathrm{Ex}^\infty \Tilde{W} \ar[r]
& \mathcal{L}_W \categ C
\end{tikzcd}
$$
where $\Tilde{W}$ denotes the smallest quasi-category of $\categ C$ that contains all the maps of $W$.
\end{definition}

\begin{remark}
Let us consider two $\infty$-functor $\pi_1 : \categ C \to \mathcal L^{(1)}_W\categ C$ and $\pi_2 : \categ C \to \mathcal L^{(2)}_W\categ C$ that satisfy condition (1) of the above definition and condition (2) in the case where $\categ D$ is the $\infty$-category $\categ S$ of $\infty$-groupoids. Then, these morphisms are canonically equivalent in the category
$$
\mathrm{Ho}(\text{Quasi-categories})_{\categ C /} .
$$
Indeed, Yoneda's lemma implies that the opposite of $\mathcal L^{(1)}_W\categ C$ and of $\mathcal L^{(2)}_W\categ C$ identify with the same sub-$\infty$-category of $\Funinfty{\categ C}{\categ S}$. Then, necessarily, $\pi_1$ and $\pi_2$ are equivalent to $\pi$. This argument \`a la Yoneda may be used to describe the mapping spaces of $\mathcal L_W\categ C$; see for instance \cite{Nuiten16}.
\end{remark}

\begin{remark}
Actually, Barwick and Kan have shown that any $\infty$-category is equivalent to the localisation of a category; see \cite{BARWICK201242}.
\end{remark}

\subsection{Localisation of simplicial and cubical categories}

\begin{definition}
Let $\categ C$ be a simplicial category (resp. a cubical category) and let $W$ be a set of morphisms. A simplicial (resp. cubical) functor $\categ C \to \categ D$ is a localisation with respect $W$ if the morphism of quasi-categories
$$
\mathcal N (\categ C) \to \mathcal N (\categ D)
$$
is a localisation with respect to $W$.
\end{definition}

\begin{lemma}\label{lemmalift}
Let us consider a fibrant simplicial category $\categ D$ and a morphism $f \in \categ D$. Then, $f$ is an equivalence in $\categ D$ if and only if the following square diagram has a lifting
$$
\begin{tikzcd}
{[1]} 
\ar[r, "f"] \ar[d]
& \categ D
\ar[d]
\\
\mathfrak C N (\mathrm{Iso})
\ar[r]
& \ast.
\end{tikzcd}
$$
\end{lemma}

\begin{proof}
This is a straightforward consequence of \cite[Lemma 2.4]{Bergner07}.
\end{proof}

\begin{definition}
Let $\categ C$ be a simplicial category and let $W$ be a set of morphisms. Then, we denote $\mathbb L_W \categ C$ the following pushout in the category of simplicial categories
$$
\begin{tikzcd}
\coprod_{f \in W} {[1]} 
\ar[r] \ar[d]
& \categ C
\ar[d]
\\
\coprod_{f \in W} \mathfrak C N (\mathrm{Iso})
\ar[r]
& \mathbb L_W \categ C .
\end{tikzcd}
$$
Since the Bergner model structure is left proper, this is an homotopy pushout.
\end{definition}

\begin{proposition}
The map $\pi : \categ C \to \mathbb L_W \categ C$ is a localisation
\end{proposition}

\begin{proof}
This is a consequence of the facts that the adjunction $\mathfrak C \dashv N$ is a Quillen equivalence (when simplicial sets are endowed with the Joyal model structure) and that the puhsout defining $\mathbb L_W \categ C$ is an  homotopy pushout.
\end{proof}

\begin{definition}
Let $\categ C$ be a category and let $W$ be a subcategory. The hammock localisation of $\categ C$ with respect to $W$ defined in \cite{DwyerKan80b} is denoted $\hammock \categ C$.
\end{definition}

\begin{theorem}\cite{DwyerKan80b}\cite{Hinich16}
Given a category $\categ C$ and a subcategory $W$, the simplicial functor
$$
\categ C \to \hammock \categ C
$$
is a localisation with respect to $W$.
\end{theorem}

\begin{proof}[Idea of the proof]
The hammock localisation is equivalent to another construction called the standard localisation which can easily be shown to be a model for the localisation.
\end{proof}

\begin{corollary}
The functor $\categ C \to L^H_W \categ C \to \mathrm{Ex}^\infty(L^H_W \categ C)$ factorises as
$$
\categ C \to \mathbb L_W \categ C \to
\mathrm{Ex}^\infty(L^H_W \categ C)
$$
where the second map is an equivalence of simplicial categories that is bijective on objects.
\end{corollary}

\begin{proof}
The fact that such a factorisation exists follows from Lemma \ref{lemmalift}. The fact that the second map is an equivalence follows from the fact that the composite functor is a localisation of $\categ C$ with respect to $W$.
\end{proof}

\subsection{Localisation of model categories}

In this subsection, we recall some results of Dwyer and Kan about localisation of model categories.

Let $\categ M$ be a model category. We denote $W, U,V$ the wide subcategories of respectively weak equivalences, acyclic cofibration and acyclic fibrations and we denote $\categ M_{c}, \categ M_{f}, \categ M_{cf}$ the full subcategories of respectively cofibrant objects, fibrant objects and fibrant-cofibrant objects.

\begin{theorem}\cite{DwyerKan80c}\cite{Hinich16}
The following functors
$$
\begin{tikzcd}
M_{cf}
\ar[r] \ar[d]
& M_c
\ar[d]
\\
M_f
\ar[r]
& M
\end{tikzcd}
$$
yield equivalences of  $\infty$-categories after localisation with respect to weak equivalences.
\end{theorem}

\begin{definition}
For any two objects $X,Y$ of $\categ M$, let $W^{-1}\categ M(X,Y)$ be the category whose objects are spans in $\categ M$
$$
X \leftarrow Z \to Y
$$
whose left arrow belongs to $W$. The morphisms are just morphisms of spans.
One can define similarly $V^{-1}\categ M(X,Y)$ 
\end{definition}

\begin{theorem}\cite{DwyerKan80b}\cite{DwyerKan80c}\label{thmmappingspaceloc}
Let $X,Y$ be two objects of $\categ M$. If $Y$ is fibrant, then the canonical morphisms of simplicial sets
$$
N(V^{-1}\categ M(X,Y)) \to
N(W^{-1}\categ M (X,Y)) \to\hammock\categ M(X,Y)
$$
are equivalences.
\end{theorem}

\begin{remark}
Regarding the previous theorem, the reader might have a look at \cite{Cisinski10} and \cite{Nuiten16} which treat the case of categories of fibrant objects and some $\infty$-categorical generalisations.
\end{remark}

\begin{definition}
For any object $X$, let $\categ M \downarrow_V X$ be the category whose objects are acyclic fibrations $f:Z \to X$ targeting $X$ and whose morphisms from $f: Z \to X$ to $f' : Z' \to X$ are morphisms $g :Z \to Z'$ so that $f = f'g$. We also denote $\categ M_c \downarrow_V X$ the full subcategory of $\categ M \downarrow_V X$ spanned by the acyclic fibrations $f:Z \to X$ whose source $Z$ is cofibrant.
\end{definition}

\begin{proposition}\cite{DwyerKan80c}\label{propfinal}
Let $X$ be an object of $\categ M$ and let $(X_n)_{n \in \square}$ be a Reedy cofibrant resolution of $X$ so that the maps $X_n \to X$ are acyclic fibrations. Then the functors
$$
\square^\op \xrightarrow{n \mapsto (X_n \to X)} (\categ M_c \downarrow_V X)^\op
\to (\categ M \downarrow_V X)^\op
$$
are both homotopically cofinal.
\end{proposition}

\begin{remark}
Actually, Dwyer and Kan proved that the composite map is homotopically cofinal (actually, they used $\Delta$ instead of $\square$ but the proof adapts mutatis mutandis). But the same arguments show that the first map is also homotopically cofinal. Since the first map and the composite map are both homotopically cofinal, so is the second map.
\end{remark}

\begin{lemma}
We have a canonical isomorphism of categories
$$
V^{-1}\categ M(X,Y) =
\int^t_{Z \in \categ M \downarrow_V X} \hom_{\categ M}(Z,Y) .
$$
\end{lemma}

\begin{proof}
Straightforward.
\end{proof}

\subsection{Detecting localisation}

Let $\categ M$ be a model category.

\begin{lemma}\label{lemmacolimhammock}
For any two objects $X,Y$ where $Y$ is fibrant, the following morphism
$$
\varinjlim_{Z \in (\categ M_{c} \downarrow_V X)^\op}^h \hom_{\categ M_{cf}}(Z,Y)
\to \varinjlim_{Z \in (\categ M_{c} \downarrow_V X)^\op}^h \mathbb L_W\categ M_{cf}(Z,Y)
$$
is an equivalence.
\end{lemma}

\begin{proof}
We have a morphism of simplicial functors from $(\categ M_{c} \downarrow_V X)^\op$ to simplicial sets
from
$ L^H_W\categ M_{cf}(-,Y)
$
to the constant functor 
$
L^H_W\categ M_{cf}(X,Y)$
that consists for an element $(Z, \phi)$ in composing with $\phi^{-1}$. The resulting map
$$
\varinjlim_{(\categ M_{c} \downarrow_V X)^\op}^h \hom_{\categ M_{cf}}(-,Y)
\to \varinjlim_{(\categ M_{c} \downarrow_V X)^\op}^h L^H_W\categ M_{cf}(-,Y)
\to \varinjlim_{(\categ M_{c} \downarrow_V X)^\op}^h L^H_W\categ M_{cf}(X,Y)
\to L^H_W\categ M_{cf}(X,Y)
$$
is the equivalence described in Theorem \ref{thmmappingspaceloc}. Moreover, the middle arrow is an equivalence as well as the right arrow since the category $(\categ M_{c} \downarrow_V X)^\op$ is contractible. Hence, the first arrow is an equivalence. One can conclude by applying the 2-out-of-3 rule to the following square diagram
$$
\begin{tikzcd}
\varinjlim_{(\categ M_{c} \downarrow_V X)^\op}^h \hom_{\categ M_{cf}}(-,Y)
\ar[r] \ar[d]
& \varinjlim_{(\categ M_{c} \downarrow_V X)^\op}^h L^H_W\categ M_{cf}(-,Y)
\ar[d]
\\
\varinjlim_{(\categ M_{c} \downarrow_V X)^\op}^h \mathbb L_W\categ M_{cf}(-,Y)
\ar[r]
& \varinjlim_{(\categ M_{c} \downarrow_V X)^\op}^h \mathrm{Ex}^\infty L^H_W\categ M_{cf}(-,Y).
\end{tikzcd}
$$
\end{proof}

\begin{theorem}\label{thmcondition}
Let us consider a functor between simplicial categories (or a functor between cubical categories) $\rho : \categ M_{cf} \to \categ D$. Let us suppose that
\begin{enumerate}
    \item the functor $h(\rho)$ sends morphisms in $W$ to isomorphisms;
    \item the functor $h(\rho)$ is essentially surjective on objects;
    \item for any two objects $X,Y$ in $\categ M_{cf}$, the map
    $$
    \varinjlim_{Z \in (\categ M_{c} \downarrow_V X)^\op}^h \hom_{\categ M_{cf}}(Z,Y) \to  \varinjlim_{Z \in (\categ M_c \downarrow_V X)^\op}^h \categ D(Z,Y)
    $$
    is an equivalence.
\end{enumerate}
Then, $\rho$ is a localisation.
\end{theorem}

\begin{proof}
Let us suppose that $\categ D$ is a simplicial category (in the cubical case, it suffices to replace $\categ D$ by $L_\Delta(\categ D)$).
We can suppose that
\begin{itemize}
    \itemt $\categ D$ has the same set of objects as $\categ M_{cf}$, and $\rho$ is the identity on objects; if it is not the case, we can replace $\categ D$ by $\categ D'$ whose set of objects is $\Ob(\categ M_{cf})$ and so that
    $$
    \categ D'(x,y) = \categ D(\rho(x), \rho(y));
    $$
    then the canonical map $\categ D \to \categ D'$ is an equivalence of simplicial categories;
    \itemt $\categ D$
    is fibrant; if it is not the case, we can replace $\categ D$ by $\mathrm{Ex}^\infty(\categ D)$.
\end{itemize}

By condition (1), the simplicial functor ${\rho}$
factorises as
$$
\categ M_{cf} \xrightarrow{\pi} \mathbb{L}_W \categ M_{cf}
\xrightarrow{\xi} \categ D .
$$

Let $X, Y$ be two objects in $\categ M_{cf}$. Let us consider the following diagram of simplicial sets
$$
\begin{tikzcd}
\varinjlim_{ (\categ M_{c} \downarrow_V X)^\op}^h\categ M_{cf}(, Y) 
\ar[r] \ar[d]
& \varinjlim_{ (\categ M_{c} \downarrow_V X)^\op}^h \mathbb L_W\categ M_{cf}(-, Y)
\ar[d]
\\
\varinjlim_{ (\categ M_{c} \downarrow_V X)^\op}^h \categ D(-, Y)
\ar[r, equal]
& \varinjlim_{ (\categ M_{c} \downarrow_V X)^\op}^h \categ D(-, Y)
\end{tikzcd}
$$
The left vertical arrow is an equivalence by hypothesis. The top horizontal arrow is an equivalence by Lemma \ref{lemmacolimhammock}. Hence, the right vertical arrow is also an equivalence.

Now, let us consider the following square diagram
$$
\begin{tikzcd}
\mathbb L_W\categ M_{cf}(X, Y)
\ar[r] \ar[d]
&
\categ D(X, Y)
\ar[d]
\\
\varinjlim_{ (\categ M_{c} \downarrow_V X)^\op}^h \mathbb L_W\categ M_{cf}(X, Y)
\ar[r]
\ar[d]
&  \varinjlim_{ (\categ M_{c} \downarrow_V X)^\op}^h \categ D(X, Y)
\ar[d]
\\
\varinjlim_{ (\categ M_{c} \downarrow_V X)^\op}^h \mathbb L_W\categ M_{cf}(-, Y)
\ar[r]
& \varinjlim_{ (\categ M_{c} \downarrow_V X)^\op}^h \categ D(-, Y) 
\end{tikzcd}
$$
where the top vertical maps are induced by the inclusion of $\{X\}$ into $(\categ M_c \downarrow_V X)^\op$. These top vertical arrows are weak equivalences since $(\categ M_{c} \downarrow_V X)^\op$ is contractible. The lower vertical maps are equivalences since they proceeds from equivalences of diagrams. Hence, by the 2-out-of-3 rule, the top horizontal map is a weak equivalence. Thus, the simplicial functor $\xi$ is an equivalence.
\end{proof}

\subsection{Localisation of squared model categories}

Let $\categ M$ be a model category equipped with an homotopical left square structure $Q$.

\begin{lemma}\label{lemmacubicalapprox}
Let $X$ be an object of $\categ M$ and let $(X_n)_{n \in \square}$ be a Reedy cofibrant resolution of $X$. Then, for any fibrant object $Y$ the canonical functor
$$
\square \downarrow \hom_{\categ M}(X_-, Y) \to W^{-1}\categ M (X,Y)
$$
is an homotopy equivalence of categories.
\end{lemma}

\begin{proof}
One can factorise the morphism $(X_n)_{n \in \square} \to X$ of cocubical objects in $\categ M$ into a Reedy acylic cofibration followed by a Reedy acylic fibration
$$
(X_n)_{n \in \square} \to (X'_n)_{n \in \square} \to X .
$$
In particular $(X'_n)_{n \in \square}$ is Reedy cofibrant and the maps $X'_n \to X$ are acyclic fibrations for any $n\geq 0$.

On the one hand, the following diagram of categories commutes up to a canonical natural transformation induced by the maps $X_n \to X'_n$
$$
\begin{tikzcd}
\square \downarrow \hom_{\categ M}(X'_-, Y)
\ar[r] \ar[d]
& \square \downarrow \hom_{\categ M}(X_-, Y)
\ar[d]
\\
V^{-1}\categ M (X,Y)
\arrow[ru, Rightarrow,shorten >= 10pt, shorten <=10pt]
\ar[r]
& W^{-1}\categ M (X,Y) .
\end{tikzcd}
$$
We can notice that the bottom horizontal map is an homotopy equivalence.

On the other hand, since the map $(X_n)_{n \in \square} \to (X'_n)_{n \in \square}$ is a Reedy acyclic cofibration, then the morphism of cubical sets $(\hom_{\categ M}(X'_n,Y))_{n \in \square^\op}\to (\hom_{\categ M}(X_n,Y))_{n \in \square^\op}$ is an acylic fibration. Thus, the functor $\square \downarrow \hom_{\categ M}(X'_-, Y) \to \square \downarrow \hom_{\categ M}(X_-, Y)$ is an homotopy equivalence.

To conclude, it suffices to prove that the morphism $\square \downarrow \hom_{\categ M}(X'_-, Y)
\to
V^{-1}\categ M (X,Y)$ is an equivalence, which follows from Proposition \ref{propfinal}.
\end{proof}

\begin{theorem}
The functor of cubical categories
$$
\categ M_{cf} \to \categ M_{cf}^Q
$$
is a localisation with respect to weak equivalences.
\end{theorem}

\begin{proof}
It suffices to show that this functor satisfies the conditions of Theorem \ref{thmcondition}. The two first conditions are straightforward to check. Thus, let us show that the third condition is satisfied. By Corollary \ref{corollaryhomotopycolimitcube}, this amounts to prove that for any two fibrant-cofibrant objects $X,Y$ of $\categ M$ the functor
$$
\int_{Z \in  (\categ M_c \downarrow_V X)}^t \hom_{\categ M_{cf}}(Z,Y) \to \int_{Z \in  (\categ M_c \downarrow_V X)}^t \square \downarrow \hom_{\categ M}(Q_- Z,Y)
$$
that sends an element $(Z, f : Z \to Y)$ to $(Z, 0, f \circ \beta(Z) : Q_0(Z) \to Y)$ is an homotopy equivalence.

We have a functor
$$
 \int_{Z \in  (\categ M_c \downarrow_V X)}^t \square \downarrow \hom_{\categ M}(Q_- Z,Y)
\to W^{-1} \categ M(X,Y)
$$
that 
\begin{itemize}
    \itemt sends an object $(f: Z \to X, n, g : Q_n(Z) \to Y)$ to the span
$$
(X \leftarrow Q_n Z \to Y) \in W^{-1} \categ M (X,Y);
$$
\itemt sends a morphism from $(Z,n,Q_n(Z) \to Y)$ to $(Z',m,Q_m(Z') \to Y)$ given by a morphism $f :  Z \to Z'$ and a map $\phi : \square[n] \to \square[m]$ to the morphism of spans induced by the following diagram
$$
\begin{tikzcd}
&& Q_n(Z)
\ar[ld] \ar[rdd] \ar[d,"Q_\phi"]
\\
& Z
\ar[ld] \ar[d,"f"]
& Q_m(Z)
\ar[d,"Q_m(f)"] \ar[l] \ar[rd]
\\
X
& Z'
\ar[l]
& Q_m(Z')
\ar[l] \ar[r]
& Y .
\end{tikzcd}
$$
\end{itemize}
This functor fits in the following diagram whose square commutes up to a canonical natural transformation induced by the maps $Q_0(Z) \to Z$
$$
\begin{tikzcd}
& \square \downarrow \hom_{\categ M}(Q_-(X), Y)
\ar[d]
\\
\int_{Z\in  (\categ M_c \downarrow_V X)}^t \hom_{\categ M_{cf}}(Z,Y) \ar[r] \ar[d]
&\int_{Z \in  (\categ M_c \downarrow_V X)}^t \square \downarrow \hom_{\categ M}(Q_- Z,Y)
\ar[d]
\\
V^{-1}\categ M (X,Y)
\arrow[ru, Leftarrow,shorten >= 10pt, shorten <=10pt]
\ar[r]
& W^{-1}\categ M (X,Y) .
\end{tikzcd}
$$
The composite right vertical map is an homotopy equivalence by Lemma \ref{lemmacubicalapprox}. Its first component is an homotopy equivalence since the category $\categ M_c \downarrow_V X$ is homotopically trivial and since the functors
$\square \downarrow \hom_{\categ M}(Q_- (Z), Y) \to \square \downarrow \hom_{\categ M}(Q_- (Z'), Y)$ induced by the maps $Z' \to Z$ are homotopy equivalences. Hence, by the 2-out-of-3 property, the second component of this map is an homotopy equivalence.
We also already know that the bottom horizontal map and the left vertical map are homotopy equivalences. Hence the top horizontal map is an homotopy equivalence.
\end{proof}

\bibliographystyle{amsalpha}
\bibliography{bib}

\end{document}